\documentclass[11pt]{article}
\usepackage{}

\usepackage{amsmath,amsthm,verbatim,amssymb,amsfonts,amscd,graphicx}

\usepackage{bm}
\usepackage{tikz}
\usetikzlibrary{arrows,shapes,positioning}
\usetikzlibrary{decorations.markings}
\tikzstyle arrowstyle=[scale=1]
\tikzstyle directed=[postaction={decorate,decoration={markings, mark=at position 0.75 with {\arrow[arrowstyle]{stealth}}}}]
\tikzstyle redirected=[postaction={decorate,decoration={markings, mark=at position 0.35 with {\arrow[arrowstyle]{stealth}}}}]

\newtheorem{theorem}{Theorem}[section]
\newtheorem{corollary}[theorem]{Corollary}
\newtheorem{definition}[theorem]{Definition}

\newtheorem{lemma}[theorem]{Lemma}

\newtheorem{proposition}[theorem]{Proposition}

\DeclareMathOperator{\supp}{{\rm supp}}

\DeclareMathOperator{\Z}{\mathbb{Z}}

\newcommand{\JCTB}{{\it J. Combin. Theory Ser. B}}

\newcommand{\JGT}{{\it J. Graph Theory}}

\newcommand{\DM}{{\it Discrete Math.}}

\newcommand{\SIAMDM}{{\it SIAM J. Discrete Math.}}

\newcommand{\CJM}{{\it Canad. J. Math.}}

\newcommand{\PLMS}{{\it Proc. London Math. Soc.}}

\newcommand{\EJC}{{\it European J. Combin.}}

\begin{document}

\title{\textbf{An 8-flow theorem for signed graphs}}

\author{
Rong Luo${}^{{}^1}$, Edita M\'a\v cajov\'a${}^{{}^2}$, 
Martin \v Skoviera${}^{{}^2}$, Cun-Quan Zhang${}^{{}^1}$
\\[4mm]
\small ${}^1$Department of Mathematics, West Virginia University,\\
Morgantown, WV 26506, USA.\\
\small \texttt{\{rluo,cqzhang\}@mail.wvu.edu}
\\[4mm]
\small ${}^2$Department of Computer Science, Comenius
University\\[-0.8ex]
\small Mlynsk\'a dolina, 842 48 Bratislava, Slovakia\\
\small \texttt{\{macajova,skoviera\}@dcs.fmph.uniba.sk}
 }

\date{\today}

\maketitle

\begin{abstract}
We prove that a signed graph admits a nowhere-zero $8$-flow
provided that it is flow-admissible and the underlying graph
admits a nowhere-zero $4$-flow. When combined with the 4-color
theorem, this implies that every flow-admissible bridgeless
planar signed graph admits a nowhere-zero $8$-flow. Our result
improves and generalizes previous results of Li et al.
(European J. Combin. 108 (2023), 103627), which state that
every flow-admissible signed $3$-edge-colorable cubic graph
admits a nowhere-zero $10$-flow, and that every flow-admissible
signed hamiltonian graph admits a nowhere-zero $8$-flow.
\end{abstract}

\section{Introduction}
The study of nowhere-zero flows on graphs began with an
observation made by Tutte in \cite{Tutte1949} that, for planar
graphs, nowhere-zero integer $k$-flows are dual to vertex
$k$-colorings. As demonstrated by Bouchet \cite{Bouchet1983},
Tutte's theory naturally extends to signed graphs. For them,
the concept of an integer flow naturally comes from the surface
duality for graphs embedded in nonorientable surfaces.

In 1983, Bouchet~\cite{Bouchet1983} proposed a conjecture that
every flow-admissible signed graph admits a nowhere-zero
$6$-flow. He also showed that the conjecture is true with $216$
instead of $6$, and Z\'yka~\cite{Zyka1987} further reduced
$216$ to $30$. Recently, DeVos et al. \cite{DLLLZZ} proved that
every flow-admissible signed graph admits a nowhere-zero
$11$-flow.

Integer flows have also been examined for various specific
families of signed graphs, such as planar, cubic, eulerian,
hamiltonian, or series-parallel signed graphs. For more
information we refer the reader to  \cite{HLLSZ2021, LLLZZ,
LLLZ-JGT2023, LLSSZ2018, LLZh-EJC2018, Kaiser2016,
Macajova2016JGT, Mac-skov2015, MS-eulerian, Raspaud2011,
Schubert2015, WLZZ, Wu2014}.

Our point of departure is the following theorem recently proved
by Li et al. \cite{LLLZZ}.

\begin{theorem}{\rm (Li et al. \cite{LLLZZ})}
\label{Th: cubic} Let $(G, \sigma)$ be a flow-admissible signed
graph.
\begin{itemize}
\item[\rm (1)] If $G$ is cubic and $3$-edge-colorable, then
    $(G,\sigma)$ admits a nowhere-zero $10$-flow.
\item[\rm (2)] If $G$ is planar and bridgeless, then
    $(G,\sigma)$ admits a nowhere-zero $10$-flow.
\item[\rm (3)] If $G$ is hamiltonian,  then $(G,\sigma)$
    admits a nowhere-zero $8$-flow.
\end{itemize}
\end{theorem}

In this paper, we improve and generalize  Theorem~\ref{Th:
cubic} as follows.

\begin{theorem}
\label{Th:4-8-flow} Let $(G, \sigma)$ be a flow-admissible
signed graph.  If $G$ admits a nowhere-zero $4$-flow, then
$(G,\sigma)$ admits a nowhere-zero $8$-flow.
\end{theorem}

It is well known that $3$-edge-colorable cubic graphs,
hamiltonian graphs, and (by the $4$-color theorem) bridgeless
planar graphs all admit a nowhere-zero $4$-flow. Thus, we have
the following corollary.

\begin{corollary}
Let $(G,\sigma)$ be a flow-admissible signed graph. Then
$(G,\sigma)$ admits a nowhere-zero $8$-flow if $G$ satisfies
any of the following:
\begin{itemize}
 \item[\rm (1)] $G$ is cubic and $3$-edge-colorable,

 \item[\rm (2)] $G$ is planar and bridgeless, or

 \item[\rm (3)] $G$ is hamiltonian.
 \end{itemize}

\end{corollary}

\section{Terminology and notation}
\label{SEC:notation} Graphs considered in this paper are finite
and may have multiple edges or loops. For terminology and
notation not defined here we refer the reader to \cite{BM2008,
West1996}.

Let $G$ be a graph. For a subset $U\subseteq V(G)$ let
$\delta_G(U)$ denote the set of edges with one endvertex in $U$
and the other in $V(G)-U$. A \emph{circuit} is a connected
$2$-regular graph. A path is \emph{nontrivial} if it contains
at least two vertices.

A \emph{signed graph} $(G,\sigma)$ is a graph $G$ endowed with
a \emph{signature} $\sigma\colon E(G)\to\{+1,-1\}$. An edge $e$
with $\sigma(e) =1$ is \emph{positive} and that with $\sigma(e)
=-1$ is \emph{negative}. An \emph{all-positive signed graph} is
a  signed graph with no negative edges; such a signed graph can
be identified with a graph without a signature. A circuit in
$(G,\sigma)$ is \emph{balanced} if it contains an even number
of negative edges and is \emph{unbalanced} otherwise. A signed
graph is called \emph{balanced} if it contains no unbalanced
circuit. A signed graph is \emph{unbalanced} if it is not
balanced. Two signed graphs $(G,\sigma)$ and $(G,\sigma')$ are
\emph{equivalent} if the sets of balanced circuits with respect
to $\sigma$ and $\sigma'$ are the same. Note that a signed
graph is balanced if and only if it is equivalent to an
all-positive signed graph.

A natural way of producing equivalent signed graphs is by
switching.  \emph{Switching} a signed graph $(G,\sigma)$ at a
vertex $u$ means reversing the signs of all edges incident with
$u$ except for the loops. It can be easily seen that switching
at a vertex does not change the parity of the number of
negative edges in a circuit, and hence the resulting signed
graph $(G,\sigma')$ is equivalent to $(G,\sigma)$. Conversely,
two signed graphs with the same underlying graph are equivalent
if and only if their signatures can be obtained from each other
by a sequence of switches.

Next, we deal with the concept of an orientation of a signed
graph, which is necessary for introducing the concept of a flow
on a signed graph. We think of each edge  $e=uv$ of a signed
graph as consisting of two half-edges $h_e^u$ and $h_e^v$,
where $h_e^u$ is incident with $u$ and $h_e^v$ is incident with
$v$ (possibly $u=v$). Let $H_G(v)$ be the set of all half-edges
incident with $v$ (we omit the subscript $G$ if no confusion
arises), and let $H(G)$ be the set of all half-edges of
$(G,\sigma)$. An {\em orientation} of $(G,\sigma)$ is a mapping
$\tau\colon H(G)\rightarrow \{+1,-1\}$ such that $\tau
(h_e^u)\tau (h_e^v)=-\sigma(e)$ for each edge $e$ of $G$. A
half-edge $h_e^u$ with $\tau(h_e^u)=1$ is interpreted as
\emph{oriented away from $u$} and that with $\tau(h_e^u)=-1$ as
\emph{oriented towards~$u$}.  A signed graph $(G,\sigma)$
equipped with an orientation $\tau$ is called an {\em oriented
signed graph}.  It may be worth mentioning that the concept of
an oriented signed graph is equivalent to the concept of a
\emph{bidirected graph} (see for example \cite{Bouchet1983,
Xu2005, Zyka1987}).

\begin{definition}\label{DE: Flow}
{\rm Let $(G,\sigma)$ be a signed graph and $\tau$ be an
orientation of $(G, \sigma)$. Let  $f\colon E(G) \to A$ a
mapping where $A$ is an abelian group.
\begin{itemize}
\item[(1)] The pair $(\tau,f)$ is called an \emph{$A$-flow}
    if for every vertex $v$ one has $\sum_{h\in
    H(v)}\tau(h)f(e_h) = 0$. If the orientation $\tau$ is
    known from the context, we simply write $f$ instead of
    $(\tau, f)$.

\item[(2)] The \emph{support} of $f$, denoted by $\supp
    (f)$, is the set of all edges $e$ with $f(e) \not = 0$.

\item[(3)]  An $A$-flow $(\tau,f)$ is said to be
    \emph{nowhere-zero}  if $\supp (f)=E(G)$.

\item [(4)]  If $A = \mathbb{Z}$ and $k\ge 2$ is an
    integer such that $|f(e)|< k$ for each edge
    $e$,  the $A$-flow $(\tau, f)$ is called a
    \emph{$k$-flow}.

\item [(5)]  For an arbitrary mapping $f\colon E(G)\to
    {\mathbb Z}$ we set $E_{f=\pm i}=\{e\in E(G) :
    |f(e)|=i\}.$
\end{itemize}
}
\end{definition}

Technically, switching changes the flows on a signed graph,
nevertheless, it only reverses the directions of the half-edges
incident with the vertex where switching is performed while the
directions of other half-edges and the flow values of all edges
remain the same.

\begin{definition}
{\rm A signed graph is \emph{flow-admissible} if it admits a
nowhere-zero integer $k$-flow for some integer $k\ge 2$.

It is
easy to see that the flow-admissibility of a signed graph is
invariant under switching. }
\end{definition}

Flow-admissible signed graphs can be conveniently described by
using the concept of a signed circuit.

\begin{definition}
{\rm A \emph{signed circuit} is defined as a signed graph of
any of the following two types:
\begin{itemize}
\item[(1)] a balanced circuit;
\item[(2)] a \emph{barbell}, which itself may be one of two
    kinds -- either the union of two unbalanced circuits
    that meet at a single vertex (\emph{short barbell}) or
    the union of two vertex disjoint unbalanced circuits
    connected by a path (\emph{long barbell}).
\end{itemize}
}
\end{definition}

The following characterization of flow-admissible signed graphs
appears in \cite{MS-eulerian}; the equi\-valence of Statements 
(1) and (2) is well known and can be traced back to Bouchet 
\cite{Bouchet1983}.

\begin{proposition}\label{Le:Flow-amissible}
{\rm (M\'a\v cajov\'a and \v Skoviera~\cite{MS-eulerian})} The
next three statements are equivalent for every connected signed 
graph $(G,\sigma)$.
\begin{itemize}

\item[{\rm (1)}] $(G,\sigma)$ is flow-admissible.

\item[{\rm (2)}] Each edge of $G$ is contained in a signed
    circuit of $(G,\sigma)$.

\item[{\rm (3)}] For each  edge  $e$ of $G$, the subgraph
    $G-e$ has no balanced component.
\end{itemize}
\end{proposition}

\section{Lemmas}
We begin with a lemma that concerns the effect of a contraction
on a flow-admissible signed graph.
To define a contraction in a signed graph, consider an edge $e=uv$ of a
signed graph $(G,\sigma)$
with $u\ne v$. If $\sigma(e)=1$, we \emph{contract} $e$ by
identifying $u$ with $v$ and by removing all positive loops
that may arise. If $\sigma(e)=-1$, we first switch the
signature at $u$ or $v$ to make $e$ positive and proceed as
before. If we repeatedly apply the contraction process to the edges 
of a connected subgraph~$H$ of $G$, contraction will produce a
signed graph $(G/H,\sigma/H)$ where $H$ is replaced with a
single vertex $v_H$ and possibly several negative loops attached to
it; if $H$ is balanced, no negative loops will occur. To
contract a disconnected graph, we contract each component
separately. Note that the signature $\sigma/H$ is determined by
$(G,\sigma)$ and $H$ only up to equivalence. Nevertheless, it
is easy to see that if $(G,\sigma)$ and $(G,\sigma')$ are
equivalent signed graphs, then so are $(G,\sigma/H)$ and
$(G,\sigma'/H)$.

\begin{lemma}\label{Le:contraction}
Let $(G,\sigma)$ be a signed graph and let $H$ be a bridgeless
balanced subgraph of $G$. Then $(G/H,\sigma/H)$ is
flow-admissible if and only if $(G,\sigma)$ is flow-admissible.
Moreover, every nowhere-zero flow on $(G,\sigma)$ induces one
on $(G/H,\sigma/H)$.
\end{lemma}

\begin{proof}
If  $(G,\sigma)$ is flow-admissible, then clearly so is
$(G/H,\sigma/H)$. Moreover, every nowhere-zero flow on
$(G,\sigma)$, induces one on $(G/H,\sigma/H)$.

For the converse, we employ Lemma~\ref{Le:Flow-amissible}~(3).
Without loss of generality, we may assume that $G$ is
connected. We may also assume that $H$ is connected as we can
contract components one by one. Since $H$ is bridgeless, we are
thereby assuming that $H$ is $2$-edge-connected. Finally, we
may assume that $H$ is all-positive.

First of all, the conclusion is true whenever $(G,\sigma)$ is
balanced. Indeed, we only have to check that $G$ is bridgeless.
Suppose that there is a bridge $e$ in $G$. Clearly, $e$ is not
in $H$ because $H$ is bridgeless. But then $e$ is inherited
into $G/H$ as a bridge, which is impossible because $(G/H,
\sigma/H)$ is assumed to be flow-admissible, and therefore
bridgeless.

Now assume that $(G,\sigma)$ is not balanced. Then $(G/H,
\sigma/H)$ is not balanced either.  Suppose to the
contrary that $(G/H,\sigma/H)$ is flow-admissible, but
$(G,\sigma)$ is not. Then $G$ contains an edge $e$ such that
$G-e$ has a balanced component. If $e$ is not contained in $H$,
then $(G/H)-e$ has a balanced component, contradicting the
assumption that $(G/H,\sigma/H)$ is flow-admissible. Hence
$e\in E(H)$. Since $H$ is $2$-edge-connected, $H-e$ is
connected, and therefore $G-e$ is connected, too. At the same
time, $G-e$ has a balanced component, so $(G-e,\sigma)$ is
balanced. Let us switch the signature of $G$ so that $G-e$
becomes all-positive. Since $(G,\sigma)$ is not balanced,
switching results in a signed graph $(G,\sigma')$ where $e$ is
the only negative edge. Recall that $H$ is $2$-edge-connected,
so there is a circuit $C$ through $e$ in~$H$. Clearly, $C$ is
unbalanced while $(H,\sigma')$ balanced, which is impossible.
This contradiction proves that $(G,\sigma)$ is flow-admissible.
\end{proof}

We will also need the following five lemmas.

\begin{lemma}\label{LE: flow-extension}
{\rm (Li et al. \cite{LLLZZ})} Let $(G,\sigma)$ be a signed
graph and $C$ be a chordless circuit in $G$ whose edges are all
positive. Suppose that $2\leq |\delta(V(C))| \leq 3$ and $k\geq
4$ is an integer. If $(G/C,\sigma)$  has a nowhere-zero
$k$-flow $f$, then $f$ can be extended to a nowhere-zero
$k$-flow on $(G,\sigma)$.
\end{lemma}

\begin{lemma}\label{eulerian-2-flow}
{\rm (Xu and Zhang \cite{Xu2005})}  A signed graph $(G,\sigma)$
admits a nowhere-zero $2$-flow if and only if each component of
$(G,\sigma)$ is eulerian and has an even number of negative
edges.
\end{lemma}

\begin{lemma}\label{TH: 2-to-3}
{\rm (Cheng et al. \cite{CLLZ2018})} Let $(G,\sigma)$ be a
connected signed graph. If $(G,\sigma)$ admits a $\Z_2$-flow
$f_1$ such that $\supp(f_1)$ contains an even number of
odd-size components, then it admits a $3$-flow $f_2$ such that
$\supp(f_1)=E_{f_2 = \pm 1}$.
\end{lemma}

\begin{lemma}\label{HC-cover}
{\rm (Li et al. \cite{LLLZZ})} Let $C$ be an unbalanced circuit
of a signed graph $(G,\sigma)$. If $(G,\sigma)$ is
flow-admissible and $G-E(C)$ is balanced, then $(G,\sigma)$ has
a $4$-flow $f$ satisfying the following:
\begin{itemize}
\item[{\rm (1)}] $E(C) \subseteq \supp(f)$
\item[{\rm (2)}] In $H = G[\supp(f)]$, the subgraph induced
    by $\supp(f)$,  each vertex of $H-V(C)$  has degree at
    most $3$ in $H$.
\end{itemize}
\end{lemma}

The next result is well known.

\begin{lemma}\label{LE:Tutte-4-flow}
{\rm (Tutte \cite{Tutte54})} A graph admits a nowhere-zero
$4$-flow if and only it admits a nowhere-zero
$\Z_2\times\Z_2$-flow. Moreover, a cubic graph admits a
nowhere-zero $4$-flow if and only if it is $3$-edge-colorable.

\end{lemma}

\section{Proof of Theorem~\ref{Th:4-8-flow}}
In this section, we provide a proof for
Theorem~\ref{Th:4-8-flow}. At first, we prove a special case of
the theorem for cubic graphs, which states that every
flow-admissible 3-edge-colorable cubic signed graph admits a
nowhere-zero 8-flow. This theorem is an improvement on Theorem
1.1 of \cite{LLLZZ} (see Statement (1) of Theorem 1.1 in this
paper), which only guarantees a 10-flow under the same
assumptions.  Nevertheless, the proof presented in \cite{LLLZZ}
provides a nowhere-zero 8-flow except in one case, which
requires a 10-flow. We construct a nowhere-zero 8-flow in the
missing case, so there is a nowhere-zero $8$-flow in all cases.
For the sake of completeness, we include the proof of the cases
where a nowhere-zero  $8$-flow was previously known.

\begin{theorem}\label{Th:8-flow-cubic}
Every flow-admissible signed graph whose underlying graph is
cubic and $3$-edge-colorable admits a nowhere-zero $8$-flow.
\end{theorem}

\begin{proof}
Let $(G,\sigma)$ be a $3$-edge-colorable cubic signed graph and
let $\tau$ be an orientation of $(G,\sigma)$. We fix $\tau$
during the entire proof and define all flows with respect to
this orientation or its restriction to a subgraph under
consideration.

Let $R$, $B$, and $Y$ be the color classes of a 3-edge-coloring
of $G$. For any pair $M_1, M_2\in\{R, B, Y\}$ let $M_1M_2$
denote the $2$-factor induced by $M_1\cup M_2$. By the
pigeon-hole principle, two of the color classes, say $R$ and
$B$, must have the same parity of the number of negative edges.
It follows that $RB$ has an even number of odd components.
There are two main cases to consider.

\medskip \noindent
{\bf Case 1.} $RB$ contains an unbalanced circuit.

By Lemma~\ref{TH: 2-to-3},  $(G,\sigma)$ has a $3$-flow $(\tau,
f_1)$ such that $RB = E_{f_1 = \pm 1}$ and $|f_1(e)| = 2$ only
if $e \in Y$.

\medskip \noindent
{\bf Subcase 1.1.} All three color classes have the same parity
of the number of negative edges.

In this case, $RY$ has an even number of unbalanced circuits. By
Lemma~\ref{TH: 2-to-3}  again,  $(G,\sigma)$ has a $3$-flow
$(\tau, f_2)$ such that $RY = E_{f_2 = \pm 1}$ and $|f_2(e)| =
2$ only if $e \in B$. Then $f = f_1 + 3f_2$  is a nowhere-zero
$9$-flow on $(G,\sigma)$. However,  $E_{f_2= \pm 2} \cap E_{f_1
= \pm 2} = \emptyset$, so $E_{f= \pm 8} = |f(e)|  = \emptyset$.
Thus $f$ is indeed a nowhere-zero $8$-flow on $(G,\sigma)$.

\medskip \noindent
{\bf Subcase 1.2.} $RY$ or $BY$ has an odd number of unbalanced
circuits.

The assumption implies that both $RY$ and $BY$ have an odd
number of unbalanced circuits. Choose an unbalanced circuit $C$
from $RB$ and set $R'=R\bigtriangleup C$ and
$B'=B\bigtriangleup C$; note that this is equivalent to
swapping colors $R$ and $B$ on $C$. Then $R'$, $B'$, and $Y$
determine a $3$-edge-coloring of $G$ such that all three color
classes have the same parity of the number of negative edges.
We are back to Subcase 1.1.

\medskip \noindent
{\bf Case 2.} $RB$ contains  no  unbalanced circuit.

By Lemma~\ref{eulerian-2-flow},  $(G,\sigma)$ has a $2$-flow
$f_3$ such that $\supp(f_3) = RB$.

\medskip \noindent
{\bf Subcase 2.1.} The  number of unbalanced circuits in $RY$
is even.

If we take  the $3$-flow $f_2$ defined in Subcase 1.1, then
$\supp(f_2) \cup \supp(f_3) = E(G)$. Thus $3f_3 + f_2$ is a
nowhere-zero $6$-flow on $(G,\sigma)$.

\medskip\noindent
{\bf Subcase 2.2.} There is only one unbalanced circuit in
$RY$.

Let $\mathcal{C}=\{C_1, \dots, C_t \}$ be the set of components
of $RY$, where $C_1$ is unbalanced and $C_2,\ldots,\penalty0
C_t$ are all balanced. We may assume that the balanced
components of $RY$ are all-positive.
Let $H$ be the signed graph obtained from $(G,\sigma)$ by
contracting $\mathcal{C} - \{C_1\}$. By
Lemma~\ref{Le:contraction}, $H$ is flow-admissible.

The vertex set of $H$ can be partitioned into two sets $K$ and
$\overline{K}$, where $K=V(C_1)$ and $\overline{K}$ is the set
of vertices that correspond to the balanced circuits of
$\mathcal{C}$. For $u\in \overline{K}$, let $C_u$ denote the
corresponding circuit of $\mathcal{C}$. Note that $C_1$ is an
unbalanced circuit of $H$.

We consider the following two subcases.

\medskip \noindent
{\bf Subcase 2.2.1.} $H$  contains an unbalanced circuit $C'$
that is edge-disjoint from $C_1$.

Since $G$ is cubic,  $C'$ is vertex-disjoint from $C_1$. Thus
there is a long barbell $Q$ in $H$ which consists of $C_1$ and
$C'$ and a path $P$ connecting them. Let $\tau_1$ be the
restriction of the orientation $\tau$ to $H$.
By Lemma~\ref{TH: 2-to-3}, there is a nowhere-zero $3$-flow
$(\tau_1, f'')$ on $Q$. Since $d_Q(u)=2$ or $d_Q(u)=3$ for
every $u\in V(Q) - V(C_1)$, the vertex $u$ corresponds to an
all-positive circuit $C_u$ in $(G,\sigma)$ with $|\delta_Q
(V(C_u))|=2$ or $3$. Hence, by Lemma~\ref{LE: flow-extension},
we can extend $f''$ to a 4-flow $f'$  on $(G,\sigma)$ with
$\supp(f')$ containing $\bigcup_{u\in V(Q)}E(C_u)\cup E(C_1)$.
Since  for each $u\in V(H) - V(Q)$, the circuit $C_u$ is
balanced, so $(G,\sigma)$ admits a  $2$-flow $\phi_u$ with
$E(C_u)=\supp(\phi_u)$. Thus $f_4=f'+\sum_{u\in V(H) -
V(Q)}\phi_u$ is a $4$-flow on $(G,\sigma)$ with $RY\subseteq
\supp(f_4)$. Therefore,  $f_3+ 2 f_4$ is a nowhere-zero
$8$-flow on $(G,\sigma)$.

\medskip \noindent
{\bf Subcase 2.2.2.} $H$  contains no unbalanced circuit  that
is edge-disjoint from $C_1$.

In this case, $H-E(C_1)$ is balanced and thus $G-E(C_1)$ is
balanced. By using switching operations, if necessary, we may
achieve that all negative edges of $(G, \sigma)$ occur in $C_1$
only. By Lemma~\ref{HC-cover},  $(G,\sigma)$ has  a $4$-flow
$f''$ such that $C_1 \subseteq \supp(f'')$ and every vertex in
$\supp(f'')-E(C_1)$ has degree at most $3$ in $H$. By
Lemma~\ref{LE: flow-extension}, we can extend $f''$ to a
$4$-flow  $f_5$  on $(G,\sigma)$ with $RY\subseteq \supp(f_5)$
in $(G, \sigma)$. Therefore,  $f_3+2f_5$ is a nowhere-zero
$8$-flow on $(G,\sigma)$.

\medskip \noindent
{\bf Subcase 2.3.} The  number of unbalanced circuits in $RY$
is odd and is at least $3$.

Let $C_1$ and $C_2$ be two unbalanced circuits in $RY$. Set
$R'=R \bigtriangleup C_1$ and $Y' = Y \bigtriangleup C_1$. Then
$R'$, $Y'$, and $B$ determine a  $3$-edge-coloring of $G$ such
that $B$ and $Y'$ have the same parity of the number of
negative edges. If $BY'$ contains an unbalanced circuit, then
we are back to Case 1.  Thus we may assume that all circuits in
$BY'$ are balanced. Let $f_6$ be a $2$-flow of $(G,\sigma)$
such that $\supp(f_6) = BY'$. Similarly, let $R'' = R
\bigtriangleup C_2$ and $Y'' = Y \bigtriangleup C_2$. Then
$R''$, $Y''$, and $B$ determine a $3$-edge-coloring of $G$
where $B$ and $Y''$ have the same parity of the number of
negative edges. As above, we may also assume that $BY''$
contains no unbalanced circuits. Let $f_7$ be a $2$-flow of
$(G,\sigma)$ such that $\supp(f_6) = BY''$. Then $f_3 +
2f_6+4f_7$ is a nowhere-zero $8$-flow on $(G,\sigma)$. This
completes the proof of Theorem~\ref{Th:8-flow-cubic}.
 \end{proof}

Now we are ready to prove Theorem~\ref{Th:4-8-flow}.

\medskip \noindent
{\bf Proof of Theorem~\ref{Th:4-8-flow}.} Let $(G,\sigma)$ be a
flow-admissible signed graph such that $G$ admits a
nowhere-zero $4$-flow. If $G$ is cubic, the conclusion of
Theorem~\ref{Th:4-8-flow} follows from
Theorem~\ref{Th:8-flow-cubic}. Assume that $G$ is not cubic. We
may assume that $G$ has no vertices of degree $2$ and contains
no positive loops, because any nowhere-zero flow on such a
signed graph easily extends to a nowhere-zero flow on a
subdivision of $G$ with possibly some positive loops added.

Since $G$ has a nowhere-zero $4$-flow, it has a nowhere-zero
$\Z_2\times \Z_2$-flow by Lemma~\ref{LE:Tutte-4-flow}; let us
denote this flow by $f$. Consider an arbitrary vertex $v$ of
degree $k\geq 4$. We intend to construct a new signed graph
$(G_1,\sigma_1)$, with fewer vertices of degree greater than
$3,$ by replacing $v$ with a suitable connected balanced
subgraph $H_v\subseteq G$. At the same time, we construct a
nowhere-zero $\Z_2\times\Z_2$-flow $g$ on $G_1$ such that the
contraction of $H_v$ to $v$ transforms $g$ to $f$.

In $G$, let $a$, $b$, and $c$ be the numbers of edges incident
with $v$ that carry flow values $(0,1)$, $(1,0)$, and $(1,1)$,
respectively, each negative loop being counted twice. Order the
half-edges with ends at $v$ as $h_1,h_2, \dots, h_k$ in such a
way that the flow values of the edges containing $h_1, h_2,
\dots, h_a$ are $(0,1)$, the flow values of the edges
containing $h_{a+1}, \dots, h_{a+b}$ are $(1,0)$, and the flow
values of the remaining $c$ edges incident with $v$
are~$(1,1)$. Since $f$ is a nowhere-zero $\mathbb{Z}_2\times
\mathbb{Z}_2$-flow, we have $a.(0,1) + b.(1,0) +
c.(1,1)=(0,0)$. Hence $a+c\equiv 0 \equiv b+c\pmod2$ and
therefore $a\equiv b\equiv c\pmod2$. In other words, $a$, $b$,
and $c$  are either all odd or all even.

We proceed to the construction of $(G_1,\sigma_1)$ and a
nowhere-zero $\mathbb{Z}_2\times\mathbb{Z}_2$-flow $g$ on its
underlying graph $G_1$. The construction depends on the
values of $a$, $b$, and $c$.

\medskip \noindent
{\bf Case 1.}   The numbers $a$, $b$, and $c$ are all odd, or
they are all even and at least one of them is zero; see
Figure~\ref{FIG:case 1}  for an illustration.

We blow up $v$ to an all-positive circuit $C_v = v_1v_2\cdots
v_kv_1$ such that each $v_i$ is the end of~$h_i$. We define a
nowhere-zero $\mathbb{Z}_2\times \mathbb{Z}_2$-flow $g$ on
$G_1$ as follows.
\begin{itemize}
\item[(i)]  Set $g(e) = f(e)$ for each edge $e$ of
    $G_1-E(C_v)$;
\item[(ii)]  If $a$, $b$, and $c$ are all odd,  then for
    each edge $v_iv_{i+1} \in E(C_v)$,  set
\[
g(v_iv_{i+1}) =
  \left\{
    \begin{array}{cl}
      (1,1)  & \mbox{if $i = 1, 3, \dots, a+b-1$,} \\
     (1,0) &  \mbox{if  $i = 2, 4, \dots, a-1,  a+b+1, a+b+3, \dots, k-2, k$, }\\
      (0,1) & \mbox{if $i = a+1, a+3, \dots, k-1$.}
    \end{array}
\right.
\]

\item[(iii)] If $a$, $b$, and $c$ are all even and at least
    one is zero, by symmetry we may assume that $c = 0$.

If $a\not = 0$ and $b\not = 0$, then for each edge
$v_iv_{i+1} \in E(C_v)$ set
   \[
g(v_iv_{i+1}) =
  \left\{
    \begin{array}{cl}
      (1,1)  & \mbox{if $i$ is even,} \\
     (1,0) &  \mbox{if  $i = 1,3,\dots, a-1$,}\\
      (0,1) & \mbox{if $i = a+1, a+3, \dots, k-1$.}
    \end{array}
\right.
\]
If either $a = 0$ or $b=0$, without loss of generality we
may assume that $b = 0$. Then for each edge $v_iv_{i+1} \in
E(C_v)$ set
        \[
g(v_iv_{i+1}) =
  \left\{
    \begin{array}{cl}
      (1,1)  & \mbox{if $i$ is even,} \\
     (1,0) &  \mbox{if  $i$ is odd.}
    \end{array}
\right.
\]
\end{itemize}


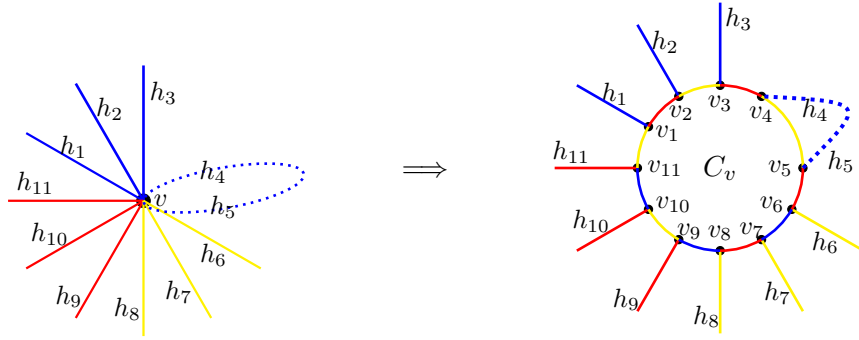
\begin{figure}
\begin{center}
\begin{tikzpicture}[scale=1.8]
\path (0:0)   coordinate (v);  \draw [fill=black] (v) circle (0.05cm);

\path (90:1) coordinate (v1);

\path (120:1) coordinate (v2);

\path (150:1) coordinate (v3);

\path (180:1) coordinate (v4);

\path (210:1) coordinate (v5);

\path (240:1) coordinate (v6);

\path (270:1) coordinate (v7);

\path (300:1) coordinate (v8);

\path (330:1) coordinate (v9);

\draw [line width=1, blue] (v)--(v1); \draw [line width=1,
blue] (v2)--(v); \draw [line width=0.85, blue] (v3)--(v);

\draw [line width=0.85, red] (v4)--(v); \draw [line
width=0.85,red] (v5)--(v);

\draw [line width=0.85,red] (v6)--(v); \draw [line
width=0.85,yellow] (v7)--(v);

\draw [line width=0.85,yellow] (v8)--(v); \draw [line
width=0.85,yellow] (v9)--(v);

\draw[dotted,line width=1,rotate=100, blue] (-0.01,-0.6)
ellipse (0.15 and 0.6);

\node[right]at (15:0) {\small $v$};

\node[right]at (150:0.8) {\small $h_1$}; \node[right]at
(122:0.8) {\small $h_2$}; \node[right]at (92:0.8) {\small
$h_3$}; \node[right]at (32:0.4) {\small $h_4$}; \node[below]at
(11:0.6) {\small $h_5$}; \node[left]at (-30:0.8) {\small
$h_6$}; \node[left]at (-57:0.8) {\small $h_7$}; \node[left]at
(-86:0.8) {\small $h_8$}; \node[left]at (-118:0.8) {\small
$h_9$}; \node[above]at (-150:0.8) {\small $h_{10}$};
\node[above]at (-178:0.8) {\small $h_{11}$};
\end{tikzpicture}\hspace{1cm}
\begin{tikzpicture}
\path (0:0)   coordinate (v); \node[above]at (0,2)
{$\Longrightarrow$};
\end{tikzpicture}\hspace{1cm}
\begin{tikzpicture}[scale=1.1]
\path (0:0)   coordinate (v);

\path (0:1) coordinate (u1); \draw [fill=black] (u1) circle
(0.05cm);

\path (60:1) coordinate (u2); \draw [fill=black] (u2) circle
(0.05cm);

\path (90:1) coordinate (u3); \draw [fill=black] (u3) circle
(0.05cm);

\path (120:1) coordinate (u4); \draw [fill=black] (u4) circle
(0.05cm);

\path (150:1) coordinate (u5); \draw [fill=black] (u5) circle
(0.05cm);

\path (180:1) coordinate (u6); \draw [fill=black] (u6) circle
(0.05cm);

\path (210:1) coordinate (u7); \draw [fill=black] (u7) circle
(0.05cm);

\path (240:1) coordinate (u8); \draw [fill=black] (u8) circle
(0.05cm);

\path (270:1) coordinate (u9); \draw [fill=black] (u9) circle
(0.05cm);

\path (300:1) coordinate (u10); \draw [fill=black] (u10) circle
(0.05cm);

\path (330:1) coordinate (u11); \draw [fill=black] (u11) circle
(0.05cm);

\path (30:2) coordinate (v1);

\path (60:2) coordinate (v2);

\path (90:2) coordinate (v3);

\path (120:2) coordinate (v4);

\path (150:2) coordinate (v5);

\path (180:2) coordinate (v6);

\path (210:2) coordinate (v7);

\path (240:2) coordinate (v8);

\path (270:2) coordinate (v9);

\path (300:2) coordinate (v10);

\path (330:2) coordinate (v11);

\node at (0:0) {$C_v$};

\node[right]at (150:1.8) {\small $h_1$}; \node[right]at
(122:1.8) {\small $h_2$}; \node[right]at (92:1.8) {\small
$h_3$}; \node[below]at (40: 1.5) {\small $h_4$}; \node[below]at
(13:1.5) {\small $h_5$}; \node[left]at (-30:1.8) {\small
$h_6$}; \node[left]at (-57:1.8) {\small $h_7$}; \node[left]at
(-86:1.8) {\small $h_8$}; \node[left]at (-118:1.8) {\small
$h_9$}; \node[above]at (-150:1.8) {\small $h_{10}$};
\node[above]at (-178:1.8) {\small $h_{11}$};

\node[right]at (155:1) {\small $v_1$}; \node[below]at (118:1)
{\small $v_2$}; \node[below]at (90:1) {\small $v_3$};
\node[below]at (60: 1) {\small $v_4$}; \node[left]at (0:1)
{\small $v_5$}; \node[left]at (-26:1) {\small $v_6$};
\node[left]at (-49:1) {\small $v_7$}; \node[above]at (-90:1)
{\small $v_8$}; \node[right]at (-130:1) {\small $v_9$};
\node[right]at (-154:1) {\small $v_{10}$}; \node[right]at
(-180:1) {\small $v_{11}$};

\draw [line width=1, blue] (u3)--(v3); \draw [line
width=1,blue] (u4)--(v4); \draw [line width=1,blue] (u5)--(v5);
\draw [line width=1, red] (u6)--(v6); \draw [line width=1,red]
(u7)--(v7); \draw [line width=1, red] (u8)--(v8); \draw [line
width=1,yellow] (u9)--(v9); \draw [line width=1, yellow]
(u10)--(v10); \draw [line width=1, yellow] (u11)--(v11);

\draw [line width=1,yellow]  (u3) arc(90:120:1); \draw [line
width=1,red]  (u3) arc(90:60:1);

\draw [line width=1,red]  (u4) arc(120:150:1); \draw [line
width=1,yellow]  (u5) arc(150:180:1);

\draw [line width=1,blue]  (u6) arc(180:210:1);

\draw [line width=1,yellow]  (u7) arc(210:240:1);

\draw [line width=1,blue]  (u8) arc(240:270:1);

\draw [line width=1,red]  (u9) arc(270:300:1);

\draw [line width=1,blue]  (u10) arc(300:330:1);

\draw [line width=1,red]  (u11) arc(330:360:1); \draw [line
width=1,yellow]  (u1) arc(0:60:1);

\draw [dotted,line width=1.5, blue] (u1) .. controls
(1.8,0.8)..(u2);
\end{tikzpicture}
\end{center}
\caption{$a =5$, $b = c = 3$. The flow values of blues edges,
red edges and yellow edges are $(0,1), (1,0), (1,1)$,
respectively. } \label{FIG:case 1}
\end{figure}


\medskip \noindent
{\bf  Case 2.}  All of  $a$, $b$, and $c$ are nonzero and even;
see Figure~\ref{FIG:case 2}  for an illustration.

We blow up  $v$ into an all-positive circuit $C_v =
v_1v_2\cdots v_kv_1$ such that each $v_i$ is the end of $h_i$
for each $i = 1, \dots, a-1$ and each $i = a+b+1, \dots, k$,
$v_{i}$ is the end of $h_{i+1}$ for each $i =a, \dots, a+b -1$,
and $v_{a+b}$ is the end of $h_a$. We now define a nowhere-zero
$\mathbb{Z}_2\times \mathbb{Z}_2$-flow $g$ of $G_1$ as follows.
\begin{itemize}
\item[(iv)]  Set  $g(e) = f(e)$ for each edge $e$ of
    $G_1-E(C_v)$.

\item[(v)]  For each edge $v_iv_{i+1} \in E(C_v)$, set
  \[
g(v_iv_{i+1}) =
  \left\{
    \begin{array}{cl}
    (1,1)  & \mbox{if $i = 1, 3, \dots, a+b-1$,} \\
    (1,0)  & \mbox{if $i =  2, 4, \dots, a-2, a+b, a+b +2, \dots, k$,}\\
    (0,1)  & \mbox{if $i =a, a+2, \dots, a+b-2, a+b+1, a+b+3, \dots, k-1$.}
    \end{array}
\right.
\]
\end{itemize}

It is easy to check that $g$ is a nowhere-zero
$\mathbb{Z}_2\times\mathbb{Z}_2$-flow on $G_1$ in each case.


\begin{figure}
\begin{center}
\begin{tikzpicture}[scale=1.8]
\path (0:0)   coordinate (v);  \draw [fill=black] (v) circle (0.05cm);


\path (0:1) coordinate (v1);



\path (45:1) coordinate (v2);

\path (90:1) coordinate (v3);


\path (135:1) coordinate (v4);


\path (180:1) coordinate (v5);


\path (225:1) coordinate (v6);


\path (270:1) coordinate (v7);


\path (315:1) coordinate (v8);






\draw [line width=1, blue] (v)--(v1); \draw [line width=1,
blue] (v2)--(v); \draw [line width=0.85, blue] (v3)--(v);

\draw [line width=0.85,red] (v4)--(v); \draw [line
width=0.85,red] (v5)--(v);

\draw [line width=0.85,yellow] (v6)--(v); \draw [line
width=0.85,yellow] (v7)--(v);

\draw [line width=0.85,blue] (v8)--(v);



\node[above]at (40:0.1) {$v$}; \node[right] at (135:0.8)
{\small $h_8$}; \node[right] at (90:0.8) {\small $h_1$};
\node[right] at (45:0.8) {\small $h_2$}; \node[below] at
(0:0.8) {\small $h_3$}; \node[left] at (-45:0.8) {\small
$h_4$}; \node[left] at (-90:0.8) {\small $h_5$}; \node[left] at
(-135:0.8) {\small $h_6$}; \node[above] at (-180:0.8) {\small
$h_7$};

\end{tikzpicture}\hspace{1cm}
\begin{tikzpicture}
\path (0:0)   coordinate (v); \node[above]at (0,2)
{$\Longrightarrow$};
\end{tikzpicture}\hspace{1cm}
\begin{tikzpicture}[scale=1.1]

\path (0:1) coordinate (u1);\draw [fill=black] (u1) circle
(0.05cm);



\path (45:1) coordinate (u2); \draw [fill=black] (u2) circle
(0.05cm);

\path (90:1) coordinate (u3); \draw [fill=black] (u3) circle
(0.05cm);

\path (135:1) coordinate (u4); \draw [fill=black] (u4) circle
(0.05cm);

\path (180:1) coordinate (u5); \draw [fill=black] (u5) circle
(0.05cm);

\path (225:1) coordinate (u6); \draw [fill=black] (u6) circle
(0.05cm);

\path (270:1) coordinate (u7); \draw [fill=black] (u7) circle
(0.05cm);

\path (315:1) coordinate (u8); \draw [fill=black] (u8) circle
(0.05cm);




\path (0:2) coordinate (v1);


\path (45:2) coordinate (v2);

\path (90:2) coordinate (v3);

\path (135:2) coordinate (v4);

\path (180:2) coordinate (v5);

\path (225:2) coordinate (v6);

\path (270:2) coordinate (v7);

\path (315:2) coordinate (v8);




\path (-90:1) coordinate (w);






\draw [line width=1,red]  (u1) arc(0:-45:1); \draw [line
width=1,yellow]  (u1) arc(0:45:1); \draw [line width=1,red]
(u2) arc(45:90:1); \draw [line width=1,yellow]  (u3)
arc(90:135:1); \draw [line width=1,blue]  (u4) arc(135:180:1);
\draw [line width=1,yellow]  (u5) arc(180:225:1); \draw [line
width=1,red]  (u6) arc(225:270:1); \draw [line width=1,blue]
(u7) arc(270:315:1);


\node at (0:0) {$C_v$}; \node[below] at (90:1) {\small $v_1$};
\node[left] at (40:1) {\small $v_2$}; \node[left] at (0:1)
{\small $v_3$}; \node[left] at (-38:1) {\small $v_4$};
\node[above] at (-90:1) {\small $v_5$}; \node[right] at
(-138:1) {\small $v_6$}; \node[right] at (-180:1) {\small
$v_7$}; \node[right] at (-218:1) {\small $v_8$};

\node[right] at (90:1.8) {\small $h_1$}; \node[below] at
(40:1.8) {\small $h_2$}; \node[below] at (0:1.8) {\small
$h_3$}; \node[left] at (-42:1.8) {\small $h_5$}; \node[left] at
(-90:1.8) {\small $h_6$}; \node[above] at (-138:1.8) {\small
$h_4$}; \node[above] at (-180:1.8) {\small $h_7$}; \node[right]
at (-228:1.8) {\small $h_8$};

\draw [line width=0.85, blue] (u1)--(v1); \draw [line
width=0.85, blue] (u2)--(v2); \draw [line width=1, blue]
(u3)--(v3); \draw [line width=1,red] (u4)--(v4); \draw [line
width=1,red] (u5)--(v5); \draw [line width=1, blue] (u6)--(v6);
\draw [line width=1,yellow] (u7)--(v7); \draw [line width=1,
yellow] (u8)--(v8);

\end{tikzpicture}
\end{center}
\caption{$a =4$, $b = c = 2$,  The flow values of blues edges,
red edges and yellow edges are $(0,1), (1,1), (1,0)$,
respectively.} \label{FIG:case 2}
\end{figure}
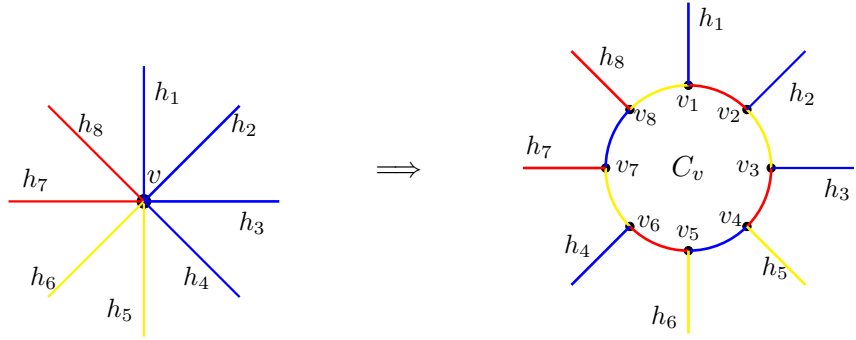

\medskip

We finish the proof. We have constructed a signed graph
$(G_1,\sigma_1)$, with a balanced subgraph $H_v\subseteq G_1$
such that $(G_1/H_v,\sigma_1/H_v)=(G,\sigma)$, and a
nowhere-zero $\mathbb{Z}_2\times\mathbb{Z}_2$-flow $g$ on
$G_1$. If $G_1$ is not cubic, we pick a vertex $v_1$ of degree
greater than $3$ and process it similarly as we did $v$ in $G$.
Eventually, after a finite number of steps, we obtain a cubic
signed graph $(G',\sigma')$, with a balanced subgraph
$H'\subseteq G'$ such that $(G'/H',\sigma'/H')=(G,\sigma)$, and
a nowhere-zero $\mathbb{Z}_2\times\mathbb{Z}_2$-flow $g'$ on
$G'$. By Lemma~\ref{LE:Tutte-4-flow}, $G'$ is
$3$-edge-colorable. Since $(G,\sigma)$ is flow-admissible, so
is $(G',\sigma')$ according to Lemma~\ref{Le:contraction}. Now
we can apply Theorem~\ref{Th:8-flow-cubic} and conclude that
$(G',\sigma')$ admits a nowhere-zero $8$-flow. After
contracting each component of $H'$ to a vertex, we finally
obtain a nowhere-zero $8$-flow on $(G,\sigma)$, as required.
\hfill$\Box$

\section*{Acknowledgements}
The first author was partially supported by a grant from 
Simons Foundation (No. 839830). The second author was partially 
supported by by the grant VEGA~1/0743/21 of the Slovak Ministry 
of Education. The third author was partially supported by 
the grant VEGA~1/0727/22 of Slovak Ministry of Education. 
Both the second and the third author were partially supported 
by the grant No.~APVV-19-0308 of the Slovak Research and 
Development Agency. 


 \end{document}